\documentclass{amsart}

\usepackage{amssymb}
\usepackage{amsmath,amsthm}
\usepackage{hyperref}
\usepackage[capitalize]{cleveref}
\usepackage{xcolor}

\newtheorem{theorem}{Theorem}[section]

\newtheorem{lemma}[theorem]{Lemma}
\newtheorem{corollary}[theorem]{Corollary}

\newtheorem{assumpx}{Theorem}
\newenvironment{NamedTheo}[2][]{\renewcommand\theassumpx{#2}\begin{assumpx}[#1]}{\end{assumpx}}

\theoremstyle{definition}
\newtheorem{definition}[theorem]{Definition}

\newtheorem{question}[theorem]{Question}

\theoremstyle{remark}
\newtheorem{remark}[theorem]{Remark}
\newtheorem{example}[theorem]{Example}

\DeclareMathOperator{\Aut}{Aut}
\DeclareMathOperator{\lcm}{lcm}
\DeclareMathOperator{\C}{C}
\DeclareMathOperator{\D}{D}
\DeclareMathOperator{\E}{E}
\DeclareMathOperator{\V}{V}
\DeclareMathOperator{\W}{W}
\DeclareMathOperator{\FQ}{\chi_F}
\DeclareMathOperator{\OQ}{\chi_O}
\DeclareMathOperator{\SW}{SW}
\renewcommand{\d}{\mathrm d}
\newcommand{\du}{\overline{\mathrm{d}}}
\newcommand{\A}{\mathcal{B}}
\newcommand{\OO}{\mathcal{O}}

\newcommand{\NP}{\mathcal{NS}}
\newcommand{\SP}{\mathcal{S}}
\newcommand{\ZZ}{\mathbb{Z}}
\newcommand{\NN}{\mathbb{N}}
\newcommand{\PP}{\mathbb{P}}

\begin{document}

\title[Density of quotient and graph orders]{Density of quotient orders in groups and applications to locally-transitive graphs}
\author{Marston Conder, Gabriel Verret, Darius Young}
\address{Department of Mathematics, University of Auckland, 1010 Auckland, New Zealand}
\email{m.conder@auckland.ac.nz, g.verret@auckland.ac.nz, darius.young@auckland.ac.nz}


\begin{abstract}
We prove that the set of orders of finite quotients of a finitely generated group has natural density $0$, $1/2$ or $1$, and characterise when each of these cases occurs. We apply this to show that the sets of orders of various families of symmetric graphs have natural density $0$.
\end{abstract}

\maketitle

\section{Introduction}\label{sec:intro}
Groups are often studied via their quotients, especially via their finite quotients. There is a huge amount of literature dedicated to studying the number $f_n^{\triangleleft }(G)$ of normal subgroups of index $n$ in a group $G$, especially in the case when $G$ is finitely generated. Determining $f_n^{\triangleleft }(G)$ exactly is often very difficult, so attention is usually restricted to the growth type of $f_n^{\triangleleft }(G)$ and how it relates to the structure of $G$. See~\cite{LubSegal} for a beautiful introduction to this topic.

One can also consider the following variant: instead of studying the set of all normal subgroups of finite index of a group $G$, one can simply consider the set of their indices, which we denote by $\FQ(G)$. (Equivalently, $\FQ(G)$ is  the set of orders of finite quotients of $G$.) As is the case for $f_n^{\triangleleft }(G)$, the set $\FQ(G)$ is often also too difficult to determine exactly, so it is natural to consider its growth or density.

This question was previously studied by Larsen who defined a particular zeta-function  $Z_G(s)=\sum_{n\in\FQ(G)} n^{-s}$ for a finitely generated group $G$ and related the  convergence of $Z_G(s)$ to the minimum dimension of a complex linear algebraic group containing an infinite quotient group of $G$, see  \cite[Theorem 0.3]{Larsen}.

In this paper, we consider the natural density $\d(\FQ(G))$ of $\FQ(G)$ in $\NN$. (The equivalent notion for not necessarily normal subgroups was previously studied by Shalev~\cite{Shalev}.) Our main result is a complete characterisation in the finitely generated case.

\begin{NamedTheo}{A}\label{Theo:A}
If $G$ is a finitely generated group, then exactly one of the following holds:
\begin{enumerate}
\item $G$ has an infinite cyclic quotient and $\FQ(G)$ consists of all positive integers, \label{Density1}
\item $G$ has an infinite dihedral quotient but no infinite cyclic quotient, and $\d(\FQ(G))=1/2$, \label{DensityHalf}
\item $G$ has no infinite cyclic or infinite dihedral quotient, and $\d(\FQ(G))=0$. \label{Density0}
\end{enumerate}
\end{NamedTheo}

In fact, our techniques are more general, and apply also to some groups that are not finitely generated. For example, we can show that if $G$ is the free product of (possibly infinitely many) cyclic groups of bounded order, at most one of which has even order, then $\d(\FQ(G))=0$, see Example~\ref{NonFGExample}. For the more general version of our findings, see Section~\ref{sec:main}.

In Section~\ref{sec:Graphs}, we apply our group-theoretic observations to various families of graphs. The most general of these are somewhat technical, but they have the following consequence:
\begin{NamedTheo}{B}\label{Theo:B}
\mbox{}
\begin{enumerate}
\item The sets of orders of finite connected $3$-valent and $5$-valent edge-transitive graphs have natural density $0$. \label{edge-transitive}
\item For every odd prime $p$, the set of orders of connected arc-transitive graphs of valency $p$ has natural density $0$. \label{arc-transitive}
\item For every integer $d\geq 3$, the set of orders of connected $2$-arc-transitive graphs of valency $d$ has natural density $0$. \label{2arc-transitive}
\end{enumerate}
\end{NamedTheo}

For the more general version, see Section~\ref{sec:Graphs}.

\section{Preliminaries}

\subsection{Number Theory}
For a subset $X$ of $\NN$, the \emph{natural density} of $X$ is 
$$\d(X)=\lim_{n\to\infty} \frac{|X\cap \{1,\ldots,n\}|}{n},$$
 if that limit exists. The \emph{upper natural density} of $X$ is 
$$\du(X)=\limsup_{n\to\infty} \frac{|X\cap \{1,\ldots,n\}|}{n},$$
which always exists.  Clearly $\d(X)=0$ if and only if $\du(X)=0$.  The following set will play a crucial role in our proof.

\begin{definition}\label{Def:Main}
For a prime $p$, let $\NP_p$ be the set of positive integers $n$ such that
\begin{enumerate}
\item $p$ divides $n$, but $p^2$ does not, and \label{CyclicSylow}
\item if $d$ divides $n$ and $d\equiv 1 \pmod p$, then $d=1$. \label{NormalSylow}
\end{enumerate}
\end{definition}

Note that (\ref{CyclicSylow}) and (\ref{NormalSylow}) in Definition~\ref{Def:Main} are  sufficient conditions on $n$ for which a direct application of the Sylow Theorems yields that every group of order $n$ has a normal Sylow $p$-subgroup of order $p$ (hence the notation $\NP_p$). 

\begin{theorem}\label{prop:density0}
For every prime $p$, we have $\d(\NP_p)=0$. 
\end{theorem}
\begin{proof}
See for example the remark following~\cite[Theorem~2.4.2]{SieveMethods}.
\end{proof}

In a vague sense, Theorem~\ref{prop:density0} says that, fixing a prime $p$, for most integers $n$ the Sylow Theorems are not enough to guarantee that a group of order $n$ has a normal Sylow $p$-subgroup of order $p$. Perhaps surprisingly, the opposite is true if the prime $p$ is allowed to vary, because $\d\left(\bigcup_{p\in\PP} \NP_p \right )=1$, where  $\PP$ denotes the set of primes (see for example~\cite[Lemma~4]{Bertram}). As we will soon see with Theorem~\ref{Theo:density1}, this remains true if $p$ is allowed to take only certain restricted values.

\begin{definition}\label{Def:Main2}
For a given positive integer $a$, let $\PP_a$ denote the set of primes $p$ such that $\gcd(a,p)=1$ and $\gcd(a,(p-1))\leq 2$, and let $\SP_a= \bigcup_{p\in\PP_a} \NP_p$.
\end{definition}

The following is an immediate consequence of~\cite[Lemma~5]{Young}.
\begin{theorem}\label{Theo:density1}
For every positive integer $a$, we have $\d(\SP_a)=1$ and thus $\d(\NN\setminus\SP_a)=0$.
\end{theorem}

Combining Theorems~\ref{prop:density0} and~\ref{Theo:density1}, one obtains the following. 

\begin{corollary}\label{Cor:density1}
If $a$ is a positive integer and $X\subseteq\NN$ with $\du(X)>0$, then the set  $\{p\in\PP_a\colon X\cap \NP_p \neq \emptyset\}$ is infinite.
\end{corollary}
\begin{proof}
Suppose on the contrary that $\{p\in\PP_a\colon X\cap \NP_p \neq \emptyset\}$ is finite. This implies that there is a finite set of primes $\PP'\subseteq\PP_a$ such that, if $p\in \PP_a\setminus \PP'$, then $X\cap \NP_p=\emptyset$. It follows that $X\subseteq (\NN\setminus\SP_a)\cup \left(\bigcup_{p\in \PP'} \NP_p\right)$.  By Theorem~\ref{Theo:density1}, $\d(\NN\setminus\SP_a)=0$ and, since $\PP'$ is finite, also  $ \d\left(\bigcup_{p\in \PP'} \NP_p\right)=0$ by Theorem~\ref{prop:density0}. This implies that $\d(X)=0$, a contradiction.
\end{proof}

\subsection{Group Theory}

For a group $G$, let $\OO(G)$ denote the subgroup of $G$ generated by elements of odd order.

\begin{lemma}\label{lem:odd}
If $N$ is a normal subgroup of a finite group $G$, then $\OO(G/N)=\OO(G)N/N$.
\end{lemma}
\begin{proof}
If the order of $g\in G$ is odd, then so is the order of $gN\in G/N$. This shows that $\OO(G)N/N\leq \OO(G/N)$. Conversely, let $gN$ be an element of odd order $k$ in $G/N$. Let $z=g^k\in N$ and let $m$ be the $2$-part of $|z|$. Since $\gcd(k,m)=1$, there exist integers $a$ and $b$ such that $1=ak+bm$. Let $h=gz^{-a}\in gN$. Since $g$ and $z$ commute, we have
$$h^k=(gz^{-a})^k=g^kz^{-ak}=z^{1-ak}=z^{bm}.$$ By definition of $m$, we know that $z^{bm}$ has odd order. Since $k$ is odd, also $h$ has odd order and therefore so does $hN=gN$. This shows that $\OO(G/N)\leq\OO(G)N/N$ and completes the proof.
\end{proof}

For a positive integer $n$, we denote the cyclic group of order $n$ by $\C_n$ and the dihedral group of order $2n$ by $\D_{2n}$.

\begin{lemma} \label{lemmaQuo}
Let $p$ be a prime. If $G$ is a finite group with $|G|\in \NP_p$, then $G$ has a quotient of the form $\C_p\rtimes H$, with $H$ acting faithfully on $\C_p$. In particular, $H$ is cyclic of order dividing $p-1$.
\end{lemma}
\begin{proof}
Since $|G|\in \NP_p$, $G$ has a normal Sylow $p$-subgroup $P$ of order $p$.  By the Schur-Zassenhaus Theorem, there exists $A\leq G$ such that $G=P\rtimes A$. Let $C=\C_G(P)$ and $K=\C_A(P)=C\cap A$. Since $P$ is abelian, $P\leq C$ and the Dedekind Modular Law yields $PK=P(C\cap A)=C\cap PA=C\cap G=C$.  Since $P\cap A=1$, also $P\cap K=1$ hence  $C=P\times K$. Since $P$ and $K$ have coprime orders, they are characteristic in $C$. Moreover, $P$ is characteristic in $G$ hence so are $C$ and $K$. Write $\overline{X}$ for $XK/K$. Note that $P\cap K=1$ hence $\overline{P}\cong P\cong\C_p$. By order considerations, we have $\overline{G}=\overline{P}\rtimes\overline{A}$. Now, $G/C$ acts faithfully by conjugation on $P$. Since $P\cap K=1$, it follows that $\overline{G}/\overline{C}=\overline{G}/\overline{P}\cong \overline{A}$ acts faithfully by conjugation on $\overline{P}$. Since $\overline{P}\cong\C_p$, it follows that $\overline{A}$ is cyclic of order dividing $p-1$. This completes the proof (with $H=\overline{A}$).
\end{proof}

\begin{definition}
We say that a group $G$ \emph{has the BCQO property} (for ``bounded cyclic quotient order'') if there exists $n\in\NN$ such that every cyclic quotient of $G$ has order at most $n$.  
\end{definition}

\begin{lemma}\label{EasyFG}
If $G$ is a finitely generated group, then either $G$ has an infinite cyclic quotient or $G$ has the BCQO property.
\end{lemma}
\begin{proof}
Since $G$ is finitely generated, so is the abelian group $G/G'$. By the classification of finitely generated abelian groups, either $G/G'$ is finite and $G$ has the BCQO property, or $G/G'$ has an infinite cyclic quotient, in which case so does $G$.
\end{proof}

\begin{lemma}\label{MediumFG}
If $G$ is a finitely generated group and there are infinitely many primes $p$ such that $G$ has a quotient isomorphic to $\D_{2p}$, then $G$ has an infinite dihedral quotient.
\end{lemma}
\begin{proof}
Let $\PP_G$ be an infinite set of primes $p$ such that $N_p$ is a normal subgroup of $G$ with $G/N_p\cong\D_{2p}$, and let $M_p$ be a normal subgroup of $G$ containing $N_p$ such that $G/M_p\cong\C_2$. Note that $M_p/N_p\cong \C_p$.

Let $G^2$ be the group generated by squares in $G$. Since $G$ is finitely generated, $G/G^2$ is finite hence the set $\{M_p:p\in\PP_G\}$ is finite. Since $\PP_G$ is an infinite set, there exists an infinite subset $\PP_G'$ of $\PP_G$ such that $M_p=M_q$ for every $p,q\in \PP_G'$.

Let $N=\bigcap_{p\in \PP_G'} N_p$. Note that $M_p/N$ is a normal subgroup of index $2$ in the infinite group $G/N$. Let $m_1,m_2\in M_p$ and $b\in G\setminus M_p$. For every $p\in \PP_G'$, we have $[m_1,m_2],b^2,m_1^bm_1\in N_p$. It follows that $[m_1,m_2],b^2,m_1^bm_1\in N$. This shows that $G/N$ is a generalised dihedral group over the abelian group $M_p/N$. In particular, every subgroup $X$ satisfying $N\leq X\leq M_p$ is normal in $G$. Since $M_p$ has index $2$ in $G$, it is finitely generated hence so is $M_p/N$.  Since $|\PP_G'|$ is infinite, so is $|M_p/N|$. As $M_p/N$ is abelian, there exists $X$ with $N\leq X\leq M_p$ such that $M_p/X$ is infinite cyclic, and  thus $G/X$ is infinite dihedral.
\end{proof}

Note that the hypothesis of finite generation in Lemmas~\ref{EasyFG} and \ref{MediumFG} is important. Without it, $\prod_{p\in\PP} \C_p$ would be a counterexample to Lemma~\ref{EasyFG} and $\left(\prod_{p\in\PP} \C_p\right)\rtimes \C_2$ (with $\C_2$ acting by inversion) would be a counterexample to Lemma~\ref{MediumFG}.

Given $a\in\NN$ and a group $G$, let $\A_a(G)$ be the subgroup of $G$ generated by elements having orders that divide $a$, and  $\OO_a(G)$ be the subgroup of $G$ generated by elements having orders that are odd and divide $a$. Clearly $\OO_a(G)$ and $\A_a(G)$ are characteristic subgroups of $G$ with $\OO_a(G)\leq \A_a(G)$.

\begin{lemma}\label{lem:BoundedOrderGen}
Let $G$ be a group. If there exists $a\in \NN$ such that $G=\A_a(G)$, then $G$ has the BCQO property.
\end{lemma}
\begin{proof}
Suppose to the contrary that $G$ has cyclic quotients of arbitrarily large order. In particular, $G$ has a normal subgroup $N$ such that $G/N$ is cyclic of order $n$ for some $n>a$. Let $M$ be the unique normal subgroup of $G$ containing $N$ such that $M/N$ is cyclic of order $\gcd(a,n)$. Since $n>a$, we know that $\gcd(a,n)<n$ and hence $M<G$. If $g$ is an element of $G$ of order dividing $a$, then $g\in M$. It follows that $\A_a(G)\leq M<G$, a contradiction.
\end{proof}

\begin{lemma}\label{lem:struc}
Let $a$ be a positive integer and let $G$ be a finite group with $|G|\in \SP_a$. If  $G=\A_a(G)$, then $G/\OO_a(G)$ is not cyclic.
\end{lemma}
\begin{proof}
By definition, there exists $p\in\PP_a$ such that $|G|\in \NP_p$. By Lemma~\ref{lemmaQuo}, $G$ has a quotient of the form $\C_p\rtimes H$, with $H$ cyclic of order dividing $p-1$ acting faithfully on $\C_p$. For $X\leq G$, write $\overline{X}$ for the image of $X$ in $\C_p\rtimes H$. Since $p\in\PP_a$, we have $\gcd(a,p)=1$ and $\gcd(a,(p-1))\leq 2$. It follows that $\overline{\OO_a(G)}=1$ and that $\overline{\A_a(G)}=\overline{G}>\C_p$. In particular, $\overline{G}$ is not cyclic, as required.
\end{proof}

\section{Main results for groups}\label{sec:main}

For a group $G$, let $\OQ(G)$ be the subset of odd integers in $\FQ(G)$.

\begin{theorem}\label{theo:mainOdd}
If $G$ is a group with the BCQO property, then  $\d(\OQ(G))=0$.
\end{theorem}
\begin{proof}
Since $G$ has the BCQO property, there exists a positive integer $a$ such that every finite cyclic quotient of $G$ has order dividing $a$. Let $\PP_G=\{p\in\PP_a\colon \OQ(G)\cap \NP_p \neq \emptyset\}$. Suppose, for a contradiction, that $\du(\OQ(G))>0$. Corollary~\ref{Cor:density1} implies that $\PP_G$ is nonempty (infinite, in fact) with say $p\in\PP_G$. By definition, we know that $G$ has a quotient whose order is odd and lies in $\NP_p$. By Lemma~\ref{lemmaQuo}, $G$ has a quotient of the form $\C_p\rtimes H_p$, with $H_p$ cyclic of odd order dividing $p-1$. It follows that $G$ has a quotient isomorphic to $H_p$, which is cyclic, so $|H_p|$ divides $a$. Since $p\in\PP_a$, we have $\gcd(a,p-1)\leq 2$ and so  $|H_p|=1$. This implies that $G$ has a quotient isomorphic to $\C_p$, which is also cyclic,  but that is a contradiction since $\gcd(p,a)=1$. It follows that $\du(\OQ(G))=0=\d(\OQ(G))$.
\end{proof}

Again note that the hypothesis that $G$ has the BCQO property is important. Otherwise, taking $G=\prod_{p\in\PP} \C_p$, we find that $\OQ(G)$ is the set of square-free odd positive integers, which has positive natural density. It is not hard to modify this example to obtain any given set of  odd positive integers closed under taking divisors as $\OQ(G)$. This yields a wild variety of possible densities.

\begin{theorem}\label{theo:main}
Let $G$ be a group with the BCQO property. If $\du(\FQ(G))>0$, then there are infinitely many primes $p$ such that $G$ has a quotient isomorphic to $\D_{2p}$.
\end{theorem}

\begin{proof}
Since $G$ has the BCQO property, there exists a positive integer $a$ such that every finite cyclic quotient of $G$ has order dividing $a$. Let $\PP_G=\{p\in\PP_a\colon \FQ(G)\cap \NP_p \neq \emptyset\}$. Since $\du(\FQ(G))>0$, Corollary~\ref{Cor:density1} implies that $\PP_G$ is an infinite set. By definition, for every $p\in \PP_G$, we know that $G$ has a quotient with order in $\NP_p$. Also for every $p\in \PP_G$, it follows from Lemma~\ref{lemmaQuo} that $G$ has a quotient of the form $\C_p\rtimes H_p$, with $H_p$ cyclic of order dividing $p-1$ acting faithfully on $\C_p$. Given $p\in \PP_G$, let $N_p$ be a normal subgroup of $G$ such that $G/N_p\cong \C_p\rtimes H_p$ as in the previous sentence, and let $M_p$ be the normal subgroup of $G$ containing $N_p$ such that $G/M_p\cong H_p$.  We know that $G/M_p$ is cyclic and so $|G/M_p|$ divides $a$. Since $p\in\PP_a$, we have $\gcd(a,p-1)\leq 2$ hence  $|G/M_p|\leq 2$. If $G=M_p$, then $G/N_p\cong \C_p$, which is a contradiction since $\gcd(p,a)=1$. It follows that $|G/M_p|=2$ and $G/N_p\cong \D_{2p}$.  Since $\PP_G$ is infinite, this completes the proof.
\end{proof}

Combining all this, we can now prove \cref{Theo:A}.

\medskip

\noindent\emph{Proof of \cref{Theo:A}.}
If $G$ has an infinite cyclic quotient, then we obtain case (\ref{Density1}). Otherwise, by Lemma~\ref{EasyFG}, $G$ has the BCQO property. If $G$ has an infinite dihedral quotient, then $\FQ(G)$ contains all even positive integers, and by Theorem~\ref{theo:mainOdd}, $\d(\OQ(G))=0$ and so  $\d(\FQ(G))=1/2$, and we obtain case (\ref{DensityHalf}).  Otherwise, $G$ has no infinite dihedral quotient, and then by Lemma~\ref{MediumFG} and Theorem~\ref{theo:main} we have $\du(\FQ(G))=0$ and we obtain case (\ref{Density0}). \qed

\begin{remark}\label{LarsenComparison}
We briefly compare our \cref{Theo:A} to Theorem 0.3 in Larsen's paper\cite{Larsen}. If a finitely generated group $G$ with infinite profinite completion has quotient dimension $n$ (see~\cite[Definition 0.2]{Larsen} for a precise definition), then Larsen's Theorem 0.3 implies that the abscissa of convergence of its zeta-function $Z_G$  (see Section~\ref{sec:intro})   is $1/n$, and $Z_G$ is singular at $1/n$. In particular, if $n\geq 2$ then $\d(\FQ(G))=0$ and in fact stronger statements can be made about the growth type of $\FQ(G)$. On the other hand, if $n=1$ then Larsen's Theorem 0.3 does not imply anything about $\d(\FQ(G))$. Indeed, the infinite cyclic group $\ZZ$, the infinite dihedral group $\ZZ\rtimes\C_2$ and $\ZZ^2\rtimes\C_4$ (the orientation-preserving symmetry group of a square tiling) all have quotient dimension $1$, but $\d(\FQ(\ZZ))=1$, $\d(\FQ(\ZZ\rtimes\C_2)=1/2$ and $\d(\FQ(\ZZ^2\rtimes\C_4)=0$. So in this case, our \cref{Theo:A} gives more information. This means that our \cref{Theo:A} and Larsen's Theorem 0.3 are each stronger than the other in different regimes of growth.
\end{remark}

In view of Remark~\ref{LarsenComparison}, we pose the following question.

\begin{question}
If $G$ is a finitely generated group, what are the possible growth types for $|\FQ(G)\cap\{1,\ldots,n\}|$ as a function of $n$?
\end{question}

Obviously \cref{Theo:A} is extremely useful when applied to finitely generated groups, but Theorems~\ref{theo:mainOdd} and~\ref{theo:main} can also be applied to many non-finitely generated groups. For example, we have the following corollary to Theorem~\ref{theo:main}.

\begin{corollary}\label{cor:struc}
Let $a\in\NN$ and let $G$ be a group with $\du(\FQ(G))>0$. If $G=\A_a(G)$, then $G/\OO_a(G)$ is not cyclic.
\end{corollary}
\begin{proof}
By Lemma~\ref{lem:BoundedOrderGen}, $G$ has the BCQO property. By Theorem~\ref{theo:main}, there are infinitely many primes $p$ such that $G$ has a quotient isomorphic to $\D_{2p}$. Choose such a prime $p$ with $p>a$ and let $N$ be a normal subgroup of $G$ such that $G/N\cong\D_{2p}$. Nontrivial elements of $\D_{2p}$ have order $2$ or $p$, so $\OO_a(G)\leq N$, and $G/\OO_a(G)$ has a quotient isomorphic to $\D_{2p}$ so cannot be cyclic.
\end{proof}

\begin{example}\label{NonFGExample}
If $G$ is the free product of (possibly infinitely many) cyclic groups of bounded order, at most one of which has even order, then $\d(\FQ(G))=0$.
\end{example}

Finally, we note that densities other than $0$, $1/2$ and $1$ are achievable for finitely generated groups if the set of finite quotients is restricted to those of various kinds. For example, if the set is restricted to quotients via torsion-free normal subgroups (usually called `smooth quotients'), then one can show that the density for the free product $\C_s \ast \C_t$ is either $0$, $1/\lcm(s,t)$ or $1/2\lcm(s,t)$.

\section{Applications to locally-transitive graphs}\label{sec:Graphs}

\subsection{Preliminaries}
All graphs and groups in this section are finite. Our terminology is fairly standard. A \emph{graph} $\Gamma$ consists of a set of vertices $\V(\Gamma)$ and a set of edges $\E(\Gamma)\subseteq\binom{V}{2}$.  An \emph{arc} of $\Gamma$ is an ordered pair of adjacent vertices. A \emph{$2$-arc} is a triple $(v_0,v_1,v_2)$ of distinct vertices such that $v_1$ is adjacent to both $v_0$ and $v_2$. 

An \emph{automorphism} of $\Gamma$ is a permutation of $\V(\Gamma)$ that preserves $\E(\Gamma)$. The automorphisms of $\Gamma$ form its \emph{automorphism group} $\Aut(\Gamma)$, under composition. We say that $\Gamma$ is \emph{vertex-transitive}  (respectively \emph{edge-transitive}, \emph{arc-transitive}, \emph{$2$-arc-transitive}) if $\Aut(\Gamma)$ is transitive on the set of vertices (respectively edges, arcs, $2$-arcs) of $\Gamma$. 


If $v$ is a vertex of $\Gamma$ and $G\leq\Aut(\Gamma)$, we denote the orbit of $v$ under the action of $G$ by $v^G$, the stabiliser  in $G$ of $v$ by $G_v$, and the permutation group induced by $G_v$ on the neighbourhood $\Gamma(v)$ of $v$ by $G_v^{\Gamma(v)}$. We say that $G$ is \emph{locally-transitive} on $\Gamma$ if $G_v^{\Gamma(v)}$ is transitive for every $v\in\V(\Gamma)$.

We collect a few results regarding edge- and local-transitivity. These are fairly standard but we include proofs for completeness.

\begin{lemma}\label{lemma:edge-transitive}
Let $\Gamma$ be a connected graph, let $\{u,v\}$ be an edge of $\Gamma$, and let $G\leq\Aut(\Gamma)$. The following hold:
\begin{enumerate}
\item $G$ is locally-transitive on $\Gamma$ if and only if $G_u^{\Gamma(u)}$ and $G_v^{\Gamma(v)}$ are transitive. \label{TwoVertices}
\item If $G$ is locally-transitive on $\Gamma$, then $G$ acts transitively on edges of $\Gamma$. \label{LocalToET}
\item If $G$ acts transitively on edges of $\Gamma$, then $\V(\Gamma)=u^G\cup v^G$. \label{ETAtMostTwoOrbits}
\item If $G$ acts transitively on edges of $\Gamma$, and $\Gamma$ is not regular of even valency, then $G$ is locally-transitive on $\Gamma$. \label{ETToLocal}
\end{enumerate}
\end{lemma}
\begin{proof}\mbox{}
\begin{enumerate}
\item Necessity is trivial. For sufficiency, let $w$ be a neighbour of $v$. Since $G_v^{\Gamma(v)}$ is transitive, $u$ is in the same $G$-orbit as $w$. It follows that $G_w^{\Gamma(w)}$ is transitive. Using this argument repeatedly together with connectedness completes the proof.
\item Consider an edge incident to $\{u,v\}$, say $\{v,w\}$ without loss of generality. Since $G$ is locally-transitive, there is an element of $G_v$ mapping $u$ to $w$, and thus mapping $\{u,v\}$ to $\{v,w\}$. This implies that every edge incident to $\{u,v\}$ is in its $G$-orbit. Using this argument repeatedly together with connectedness completes the proof.
\item Let $x\in\V(\Gamma)$. Since $\Gamma$ is connected, there is an edge $\{x,y\}$ incident with $x$. Since $G$ is transitive on edges, there exists an element of $G$ mapping $\{x,y\}$ to $\{u,v\}$, and thus mapping $x$ to one of $u$ or $v$, as required.
\item If there exists an automorphism of $G$ reversing an edge, then $G$ is transitive on arcs and locally-transitive. We now assume this is not the case. It follows that for each edge $\{x,y\}$, exactly one of $(x,y)$ and $(y,x)$ is in the $G$-orbit of $(u,v)$. Therefore the $G$-orbit of $(u,v)$ induces a directed  graph $\vec{\Gamma}$ on $\V(\Gamma)$. If $G$ is transitive on $\V(\Gamma)$, then $\vec{\Gamma}$ must be in-regular and out-regular and, by finiteness, its in-valency is equal to its out-valency, which implies that $\Gamma$ is regular of even valency, contrary to our hypothesis. Hence we may assume that $G$ is not transitive on $\V(\Gamma)$ and, by (\ref{ETAtMostTwoOrbits}), that $\V(\Gamma)$ is the disjoint union of $u^G$ and $v^G$. It follows that all arcs in $\vec{\Gamma}$ go from $u^G$ to $v^G$. Since $G$ is transitive on edges incident to $u$, we find that $G_u^{\Gamma(u)}$ is transitive. The same reasoning shows that $G_v^{\Gamma(v)}$ is transitive, and thus $G$ is locally-transitive by (\ref{TwoVertices}).
\end{enumerate}
\end{proof}

\begin{definition}[\cite{MorganSpigaVerret}]
Let $L_1$ and $L_2$ be transitive permutation groups and let $[L_1, L_2]$ denote the multiset containing $L_1$ and $L_2$. We say that $(\Gamma, G)$ is a \emph{locally-$[L_1, L_2]$} pair if $\Gamma$ is a connected graph, $G\leq\Aut(\Gamma)$, and for some edge $\{u,v\}$ of $\Gamma$, we have permutation isomorphisms $G_u^{\Gamma(u)}\cong L_1$ and $G_v^{\Gamma(v)}\cong L_2$. 

The multiset $[L_1, L_2]$ is said to be \emph{locally-restrictive} if there exists $a\in\NN$ such that for every  locally-$[L_1, L_2]$ pair $(\Gamma, G)$ and every $v\in\V(\Gamma)$, we have $|G_v|\leq a$.
\end{definition}

\subsection{Main results for locally-transitive graphs}

\begin{theorem}\label{theo:mainGraphs}
Let $L_1$ and $L_2$ be permutation groups with the following properties:
\begin{enumerate}
\item $\OO(L_1)$ and $\OO(L_2)$ are transitive, and \label{OddTransitive}
\item $[L_1, L_2]$ is locally-restrictive.
\end{enumerate}
Then the set of orders of graphs $\Gamma$ for which there exists some locally-$[L_1, L_2]$ pair $(\Gamma,G)$ has natural density $0$.
\end{theorem}
\begin{proof}
As $[L_1, L_2]$ is locally-restrictive, there exists $a\in\NN$ such that for every locally-$[L_1, L_2]$ pair $(\Gamma,G)$ and every $v\in\V(\Gamma)$, the order of $G_v$  divides $a$. 

Let $(\Gamma,G)$ be a locally-$[L_1, L_2]$ pair. By definition, we have $\OO_a(G_v)=\OO(G_v)$ for every $v\in\V(\Gamma)$. Let $\{v_1,v_2\}$ be an edge of $\Gamma$ and let $H=\langle \OO(G_{v_1}),\OO(G_{v_2})\rangle$. By the previous observation, we have $H=\langle \OO_a(G_{v_1}),\OO_a(G_{v_2})\rangle=\OO_a(H)$. It follows by Lemma~\ref{lem:struc} that $|H|\notin \SP_a$. By Theorem~\ref{Theo:density1}, the set of possibilities for $|H|$ has natural density $0$.

By Lemma~\ref{lem:odd}, $(\OO(G_{v_1}))^{\Gamma({v_1})}=\OO((G_{v_1})^{\Gamma({v_1})})\cong \OO(L_1)$ which is transitive by hypothesis. By the same argument, we find that $(\OO(G_{v_2}))^{\Gamma({v_2})}$  is also transitive. By Lemma~\ref{lemma:edge-transitive}(\ref{TwoVertices}) and (\ref{LocalToET}), $H$ is transitive on edges of $\Gamma$ and, by  (\ref{ETAtMostTwoOrbits}), $\V(\Gamma)=v_1^G\cup v_2^G$. It follows that either $|\V(\Gamma)|=|H:H_{v_1}|$ (if $v_1^G=v_2^G$) or $|\V(\Gamma)|=|H:H_{v_1}|+|H:H_{v_2}|$. Since $|H_{v_i}|\leq |G_{v_i}|\leq a$, there is only  a bounded number of possibilities for $|H_{v_i}|$. Hence for each $|H|$ there  is only  a bounded number of possibilities for $|H:H_{v_i}|$ and thus only  a bounded number of possibilities for $|\V(\Gamma)|$. Since the set of possibilities for $|H|$ has natural density $0$, the same holds for  the set of possibilities for $|\V(\Gamma)|$.
\end{proof}

Note that condition (\ref{OddTransitive}) from Theorem~\ref{theo:mainGraphs} almost always holds for quasiprimitive groups. (A nontrivial permutation group is said to be \emph{quasiprimitive} if each of its nontrivial normal subgroups is transitive.)

\begin{lemma}\label{QuasiprimitiveGenOdd}
If $L$ is a quasiprimitive group, then $\OO(L)$ is transitive unless $|L|=2$.
\end{lemma}
\begin{proof}
Since $\OO(L)$ is normal in $L$ which is quasiprimitive, $\OO(L)$ is transitive unless $\OO(L)=1$, which implies that $L$ is a $2$-group.
In that case, $L$ has a nontrivial central subgroup of order $2$ and this must be transitive (since $L$ is quasiprimitive), so $L$ has degree at most $2$ and the result follows. 
\end{proof}

It is conjectured that  $[L_1, L_2]$ is locally-restrictive if and only if $L_1$ and $L_2$ are both semiprimitive (see~\cite{MorganSpigaVerret}). (A permutation group is said to be \emph{semiprimitive} if all of its normal subgroups are transitive or semiregular.) This conjecture is still wide open but, if true, it would imply (together with Theorem~\ref{theo:mainGraphs} and Lemma~\ref{QuasiprimitiveGenOdd}) that  the set of orders of locally-quasiprimitive graphs of fixed valency at least $3$ has natural density $0$. In the meantime, using some of the known facts related to this conjecture, we can now prove \cref{Theo:B}.

\medskip

\noindent\emph{Proof of \cref{Theo:B}.}
Let $d\in\{3,5\}$, let $\Gamma$ be a connected regular edge-transitive graph of valency $d$, and let $G=\Aut(\Gamma)$. It was shown by Goldschmidt~\cite{Goldschmidt}  (when $d=3$) and Morgan~\cite{Morgan} (when $d=5$) that there exists $a\in\NN$ (not depending on $\Gamma$) such that $|G_v|\leq a$ for every $v\in\V(\Gamma)$.

Since $d$ is odd, Lemma~\ref{lemma:edge-transitive}(\ref{ETToLocal}) implies that $G$ is locally-transitive on $\Gamma$. It follows that  $(\Gamma,G)$ is a locally-$[L_1,L_2]$ pair with $L_1,L_2$ being transitive groups of degree $d$. Since $d$ is prime, $L_1$ and $L_2$ are quasiprimitive and  Lemma~\ref{QuasiprimitiveGenOdd} implies that $\OO(L_1)$ and $\OO(L_2)$ are transitive. Moreover, the results of Goldschmidt and Morgan mentioned earlier imply that  $[L_1, L_2]$ is locally-restrictive and part (\ref{edge-transitive}) follows by Theorem~\ref{theo:mainGraphs}.

Now, suppose that $\Gamma$ is as in part (\ref{arc-transitive}) or (\ref{2arc-transitive}), namely that $\Gamma$ is a connected graph and is either arc-transitive with prime valency $p\geq 3$ or $2$-arc-transitive with valency  $d\geq 3$. In this case, $\Gamma$ is vertex-transitive hence $(\Gamma,\Aut(\Gamma))$ is a locally-$[L_1,L_2]$ pair with $L_1$ and $L_2$ being permutation isomorphic. Moreover, $[L_1,L_2]$ being locally-restrictive is equivalent to $L_1$ being graph-restrictive (see~\cite{PSVGraphRestrictive,Verret}). In part (\ref{arc-transitive}), $L_1$ is a transitive group of degree $p$, whereas in part (\ref{2arc-transitive}), $L_1$ is a $2$-transitive group of degree $d$. In both cases, $L_1$ is quasiprimitive of degree at least $3$ and hence $\OO(L_1)$ is transitive by Lemma~\ref{QuasiprimitiveGenOdd}.

Finally, it follows from work of Tutte~\cite{Tutte} and Trofimov and Weiss~\cite{WeissGrass,Weiss} that transitive groups of prime degree and $2$-transitive groups are graph-restrictive. (For more details, see \cite[Section~2]{PSVGraphRestrictive}, in particular \cite[Theorem 6]{PSVGraphRestrictive} and the surrounding paragraphs.) \qed

\begin{remark}
It was shown in~\cite{PSVEnumeration} that the number of connected $2$-arc-transitive graphs of valency $3$ and order at most $n$ grows roughly like $n^{c\log n}$, for some constant $c$. Using the same approach, one can prove similar facts about the other families of graphs in \cref{Theo:B}. So while the number of possibilities for the orders up to $n$ grows slower than $n$, the number of graphs themselves grows much faster. Note that a similar phenomenon occurs for many of the ``density $0$'' conclusions in other parts of this paper. (The size of the set of orders grows slowly, but the size of the set of objects, such as distinct quotients, grows much faster.)
\end{remark}

\begin{remark} We briefly discuss the tightness of the statements in \cref{Theo:B}.
\begin{enumerate}
\item For each $n\geq 3$, there is a connected $2$-arc-transitive graph of order $n$ and valency $2$, namely the cycle graph of order $n$. This shows that the lower bound on the valency in parts (\ref{arc-transitive}) and (\ref{2arc-transitive}) of \cref{Theo:B} is necessary.
\item For $k\geq 1$ and $r\geq 3$, let $\W(k,r)$ be the graph with vertex set $\ZZ_k\times\ZZ_r$ and edge set $\{\{(x,y),(x',y+1)\}:x,x'\in\ZZ_k,y\in\ZZ_r\}$. One can check that $\W(k,r)$ is a connected arc-transitive graph of order $kr$ and valency $2k$. It follows that for fixed even $d=2k\geq 2$, the set of orders of connected arc-transitive graphs of valency $d$ has natural density at least $\frac{2}{d}$. This shows that one cannot allow even valency in part (\ref{edge-transitive}) or (\ref{arc-transitive}) of \cref{Theo:B}.
\item For $k\geq 1$ and $r\geq 2$, let $\SW(k,r)$ be the graph with vertex set $\ZZ_k\times\ZZ_r\times\ZZ_2$ and edge set $\{\{(x,y,0)(x,y,1)\}:x\in\ZZ_k,y\in\ZZ_r\}\cup  \{\{(x,y+1,0)(x',y,1)\}:x,x'\in\ZZ_k,y\in\ZZ_r\}$. One can check that $\SW(k,r)$ is a connected vertex-transitive graph of order $2kr$ and valency $k+1$. It follows that for fixed $d=k+1\geq 2$, the set of orders of connected vertex-transitive graphs of valency $d$ has natural density at least $\frac{1}{2(d-1)}$. This shows that one cannot relax from arc-transitive to vertex-transitive in part (\ref{arc-transitive}) of \cref{Theo:B}.
\end{enumerate}
\end{remark}

This leaves us with the following natural questions:
\begin{question}
Let $d$ be an odd composite positive integer. What is the natural density of the set of orders of connected arc-transitive graphs of valency $d$?
\end{question}
\begin{question}
Let $d$ be an odd integer greater than or equal to $7$. What is the natural density of the set of orders of connected edge-transitive graphs of valency $d$?
\end{question}

We are unable to resolve these questions at the moment, but we expect the answer to be $0$ in all cases.

\bigskip

{\bf Acknowledgements.} We thank Andrea Lucchini, Jeroen Schillewaert and Pablo Spiga  for some helpful discussions. The first and third authors are grateful to New Zealand's Marsden Fund for its support via grant UOA2030.

\end{document}